\documentclass[11pt,a4paper,onesided]{article}

\usepackage{ amssymb, amsmath, amsthm, graphics, xspace, enumerate}
\usepackage{graphicx}
\usepackage[a4paper,colorlinks=true,citecolor=blue,urlcolor=blue,linkcolor=blue,bookmarksopen=true,unicode=true,pdffitwindow=true]{hyperref}
 \hypersetup{pdfauthor={Nikola Sandri\'{c}}}
\hypersetup{pdftitle={A note on the Birkhoff ergodic theorem}}

\newtheorem{tm}{tm}[section]
\newtheorem{theorem}[tm]{Theorem}

\newtheorem{corollary}[tm]{Corollary}

\newtheorem{example}[tm]{Example}

\newcommand {\R} {\ensuremath{\mathbb{R}}}

\newcommand {\N} {\ensuremath{\mathbb{N}}}

\newcommand{\process}[1]{\{#1_t\}_{t\geq0}}
\newcommand{\processt}[1]{\{#1_t\}_{t\in\mathbb{T}}}
\newcommand{\chain}[1]{\{#1_n\}_{n\geq0}}
\numberwithin{equation}{section}

\begin{document}

 \title{A Note on the Birkhoff Ergodic Theorem}
 \author{Nikola Sandri\'{c}\\
Department of Mathematics\\
         Faculty of Civil Engineering, University of Zagreb, 10000 Zagreb,
         Croatia \\
        Email: nsandric@grad.hr }

 \maketitle
\begin{center}
{
\medskip

} \end{center}

\begin{abstract}
The classical Birkhoff ergodic theorem states that for an ergodic Markov process  the limiting behaviour of the time average of a   function (having finite $p$-th moment, $p\ge1$, with respect to the invariant measure)
along the trajectories of the  process,  starting from the invariant  measure,
is a.s. and in the $p$-th mean  constant and equals to the space average of the function with respect to the invariant measure. The crucial assumption here is that the process starts from the invariant measure, which is not always the case. In this paper,  under the assumptions that the underlying process is a  Markov process on  Polish space, that it admits an invariant probability measure and that its marginal distributions  converge to the invariant measure in the $L^{1}$-Wasserstein metric, we show that the assertion of
the Birkhoff ergodic theorem holds in  the $p$-th mean, $p\geq1$,  for any   bounded Lipschitz function and any initial distribution of the process.
\end{abstract}

 \smallskip


%
%
%
%


\section{Introduction}\label{s1}
One of the classical directions in the analysis of Markov processes are limit theorems
for Markov processes, such as the law of large numbers,  central limit theorem and  law of the iterated
logarithm. In this paper, we discuss a  version of the Birkhoff ergodic theorem (law of  large numbers) for a class of Markov processes.
 Let $\textbf{M}=(\Omega,\mathcal{F},\{\mathbb{P}^{x}\}_{x\in S},\processt{\mathcal{F}},\processt{\theta},  \processt{M})$
 be a (temporally homogeneous) \emph{normal Markov process} with state space $(S,\mathcal{S})$, in the sense of \cite{getoor}.
 Here,
 $S$ is a non-empty set, $\mathcal{S}$ is
a $\sigma$-algebra of subsets of $S$,  $(\Omega,\mathcal{F},\mathbb{P}^{x})_{x\in S}$ is a family of probability spaces, $\mathbb{T}$ is the time set   $\{0,1,2,\ldots\}$ or $[0,\infty)$,  $\processt{\mathcal{F}}$ is a filtration on $(\Omega,\mathcal{F})$ (non-decreasing family of sub-$\sigma$-algebras of $\mathcal{F}$) and  $\processt{\theta}$ is a family of shift operators on $\Omega$ satisfying $M_t\circ\theta_s=M_{t+s}$ for all $s,t\in\mathbb{T}$.
In the case when $\mathbb{T}=[0,\infty)$, assume that $\mathbf{M}$ is \emph{progressively measurable} (with respect to $\process{\mathcal{F}}$), that is, the map $(s,\omega)\longmapsto M_s(\omega)$ from $[0,t]\times\Omega$ to $S$ is $\mathcal{B}([0,t])\times\mathcal{F}_t/\mathcal{S}$ measurable for all $t\geq0$ (measurable in the pair of $\sigma$-algebras $\mathcal{B}([0,t])\times\mathcal{F}_t$ and $\mathcal{S}$, where $\mathcal{B}([0,t])$ denotes the Borel $\sigma$-algebra of subsets of $[0,t]$). Recall that if $S$ is a metric space and $\mathcal{S}$ is the Borel $\sigma$-algebra of  subsets of $S$, then  $\mathbf{M}$ is progressively measurable provided $t\longmapsto M_t(\omega)$ is right continuous for all $\omega\in\Omega$ (see \cite[Exercise I.6.13]{getoor}).
Further, denote by
$p^{t}(x,dy):=\mathbb{P}^{x}(M_t\in dy)$, $t\in\mathbb{T}$, $x\in S$,  the transition function of $\textbf{M}$. A  measure $\pi(dy)$ on
$\mathcal{S}$   is said to be \emph{invariant} for $\textbf{M}$ if
$$\int_{S}p^{t}(x,dy)\pi(dx)=\pi(dy),\qquad t\in\mathbb{T}.$$
A set $B\in\mathcal{F}$ is said to be
\emph{shift-invariant} (for $\textbf{M}$) if $\theta_t^{-1}B=B$ for all $t\in\mathbb{T}$. The
\emph{shift-invariant} $\sigma$-algebra $\mathcal{I}$ is a
collection of all such shift-invariant sets.
Now,  the celebrated Birkhoff ergodic theorem asserts that if a Markov process $\mathbf{M}$ admits an invariant (equilibrium) probability measure $\pi(dy)$, then
the limiting behaviour of the time average of a  function $f\in L^p(S,\pi)$, $p\ge1$,
along the trajectories of  $\textbf{M}$,  starting from  $\pi(dy)$,
  exists $\mathbb{P}^{\pi}$-a.s. and in  $L^p(\Omega,\mathbb{P}^{\pi})$, it is invariant (that is, it is measurable with respect to $\mathcal{I}$) and it is related to the space average of the function with respect to $\pi(dy)$. Here, $L^{p}(S,\pi)$ ($L^p(\Omega,\mathbb{P}^{\pi})$) denotes the space of all measurable functions $f:S\longrightarrow\R$ ($f:\Omega\longrightarrow\R$) with finite $p$-th moment with respect to $\pi(dy)$ ($\mathbb{P}^{\pi}(d\omega)$). For  a probability measure $\mu(dy)$ on $\mathcal{S}$,
$\mathbb{P}^{\mu}(d\omega)$ is defined as
$\mathbb{P}^{\mu}(d\omega):=\int_{S}\mathbb{P}^{y}(d\omega)\mu(dy).$ Also,  the collection of all probability measures on $\mathcal{S}$ is denoted by $\mathcal{P}(S)$.
\begin{theorem}{\rm(\cite[Theorems 9.6 and 9.8]{Kallenberg}).}\label{tm1.1}
Let $\textbf{M}$ be a Markov process with invariant probability measure $\pi(dy)$. Then, for any $f\in L^{p}(S,\pi)$, $p\ge1$, the following limit holds \begin{align}\label{eq1.1}\lim_{t\to\infty}\frac{1}{t}\int_{[0,t)}f(M_s)\tau(ds)=\mathbb{E}^{\pi}[f(M_0)|\mathcal{I}]\qquad \mathbb{P}^{\pi}\textrm{-a.s. and in}\ L^p(\Omega,\mathbb{P}^{\pi}),\end{align} where  $\tau(dt)$ is  the counting measure when $\mathbb{T}=\{0,1,2,\ldots\}$ and the Lebesgue measure when $\mathbb{T}=[0,\infty)$.
\end{theorem}
A Markov process
$\textbf{M}$ is said to be
 \emph{ergodic} if it possesses an invariant probability measure $\pi(dy)$ and if  $\mathcal{I}$ is
trivial with respect to $\mathbb{P}^{\pi}(d\omega)$, that is,
$\mathbb{P}^{\pi}(B)=0$ or $1$ for every $B\in\mathcal{I}$.  Now, in addition to the assumptions of Theorem \ref{tm1.1}, if  $\textbf{M}$ is ergodic, then the relation in \eqref{eq1.1} reads as follows \begin{align}\label{eq1.2}\lim_{t\to\infty}\frac{1}{t}\int_{[0,t)}f(M_s)\tau(ds)=\int_{S}f(y)\pi(dy)\qquad \mathbb{P}^{\pi}\textrm{-a.s. and in}\ L^p(\Omega,\mathbb{P}^{\pi}).\end{align}

The main assumption in Theorem \ref{tm1.1} is that  the process starts from its equilibrium, which is not always the case. For $f\in L^{p}(S,\pi)$, $p\ge1$, define $$B_f:=\left\{\omega\in\Omega:\lim_{t\to\infty}\frac{1}{t}\int_{[0,t)}f(M_s(\omega))\tau(ds)=\int_{S}f(y)\pi(dy)\right\}.$$ Clearly, $B_f\in\mathcal{I}$. Thus, in order to conclude that \eqref{eq1.2} holds $\mathbb{P}^\mu$-a.s. for any $\mu\in\mathcal{P}(S)$ one expects that $\mathcal{I}$ should be  trivial  with respect to $\mathbb{P}^{\mu}(d\omega)$ for every $\mu\in\mathcal{P}(S)$.
Indeed, in  \cite[Theorem 17.1.7]{meyn-tweedie-book} it has been proved that the following are equivalent:
\begin{itemize}
	\item [(a)] $\mathbb{P}^{\mu}(B_f)=1$  for any    $f\in L^{p}(S,\pi(dy))$ and $\mu\in\mathcal{P}(S)$;
	\item[(b)]the shift-invariant $\sigma$-algebra $\mathcal{I}$ is $\mathbb{P}^{\mu}$-trivial for every $\mu\in\mathcal{P}(S)$.
\end{itemize}

Further, note that in order to conclude (a), it is necessary that $\pi(dy)$ is a unique invariant probability measure for $\mathbf{M}$ (hence, according to \cite[Corollary 5.12]{Hairer} or \cite[Proposition 2.5]{Bhattacharya},  $\textbf{M}$ is an ergodic Markov process). Namely, if $\textbf{M}$  admits more than one invariant probability
measure, it is ergodic with respect to  at least two mutually singular invariant probability measures   (see \cite[Theorem 5.7]{Hairer}). This leads to the conclusion that, in order to  conclude (a), 
certain  additional structural properties of $\textbf{M}$ will be necessary. In  \cite[Theorem 17.1.7]{meyn-tweedie-book} it has been also proved that    (a) and (b) are equivalent to 
\begin{itemize}
	\item [(c)]  $\textbf{M}$ is a positive Harris recurrent Markov process.
\end{itemize}
Recall, a Markov process $\mathbf{M}$ is called \textit{$\varphi$-irreducible} if for the $\sigma$-finite measure $\varphi(dy)$ on $\mathcal{S}$, $\varphi(B)>0$ implies $$\int_{\mathbb{T}}p^{t}(x,B)\tau(dt)>0,\qquad x\in S.$$ The process $\mathbf{M}$ is called \textit{Harris recurrent}  if it is $\varphi$-irreducible, and $\varphi(B)>0$ implies $$ \int_{\mathbb{T}}1_{\left\lbrace M_t\in B\right\rbrace }\tau(dt)=\infty\qquad \mathbb{P}^{x}\textrm{-a.s.}$$ for all $x\in S$. Further, according to \cite[Theorem 2.6]{twedie-topological} every Harris recurrent Markov process admits a unique (up to constant multiplies) invariant (not necessary  probability) measure. If the invariant measure is finite, then the process is called \textit{positive Harris recurrent}; otherwise it is called \textit{null Harris recurrent}. In the case when the process $\mathbf{M}$ is aperiodic, in \cite[Theorem 13.0.1]{meyn-tweedie-book} and \cite[Theorem 6.1]{meynII} it has been proved that (a), (b) and (c) imply 
\begin{itemize}
	\item [(d)]$\mathbf{M}$ is strongly ergodic. 
\end{itemize}
Recall, a discrete-time Markov process $\mathbf{M}$ is called \textit{$d$-periodic} if  $d\geq1$ is the largest integer for which  there is a partition $P_1,\ldots, P_d\in\mathcal{S}$ of $S$, such that $p(x,P_{i+1}) = 1$ for all $x\in P_i$ and all $1\leq i\leq d-1$, and $p(x, P_1) = 1$ for all $x\in P_d.$ If  $d=1$, then $\mathbf{M}$ is called \textit{aperiodic}. A continuous-time Markov process $\mathbf{M}$ is called \textit{aperiodic} if it admits  an irreducible skeleton chain, that is, there is $\delta>0$ such that the discrete-time Markov process $\left\lbrace M_{\delta n}\right\rbrace_{n\geq0}$ is irreducible. 
Further, denote by  $B_b(S)$ the space of all $\R$-valued, bounded and $\mathcal{S}/\mathcal{B}(\R)$ measurable  functions on $S$, where $\mathcal{B}(\R)$ denotes the Borel $\sigma$-algebra of subsets of $\R$. Also, denote by $d_{TV}$  the \emph{total variation metric} on $\mathcal{P}(S)$, given by $$d_{TV}(\mu(dy),\nu(dy)):=\frac{1}{2}\sup_{f\in B_b(S),\, |f|_{\infty}\leq1}\left|\int_Sf(y)\mu(dy)-\int_Sf(y)\nu(dy)\right|,\qquad \mu,\nu\in\mathcal{P}(S),$$ where $|\cdot|_{\infty}:=\sup_{x\in S}|\cdot|$ denotes  the supremum norm on $B_b(S)$.
Now,  a Markov process $\textbf{M}$ is said to be \emph{strongly ergodic} if there exists  $\pi\in\mathcal{P}(S)$ such that \begin{align}\label{eq1.4}\lim_{t\to\infty}d_{TV}(p^{t}(x,dy),\pi(dy))=0,\qquad x\in S.\end{align}
Note that the relation in \eqref{eq1.4} automatically implies: (i) $\pi(dy)$ is the only measure satisfying \eqref{eq1.4}, (ii) $\pi(dy)$ is necessarily an invariant measure for $\textbf{M}$ and (iii) $\pi(dy)$ is a unique invariant measure for $\textbf{M}$. In particular,   $\textbf{M}$ is  ergodic. In other words, strong ergodicity implies ergodicity. As one could expect, ergodcity  does not in general imply strong ergodicity (see Section \ref{s3}). Moreover, under aperiodicity assumption, (d) is actually equivalent to (a), (b) and (c). Indeed,
assume that $\textbf{M}$ is strongly ergodic with invariant probability measure $\pi(dy)$.  Then, by following \cite[Proposition 2.5]{Bhattacharya},  for any $x\in S$ and $f\in L^{p}(S,\pi)$ we have that
\begin{align}\label{eq1.5}\left|\mathbb{P}^{x}(B_f)-1\right|&=\left|\mathbb{E}^{x}\left[\mathbb{P}^{M_t}(B_f)\right]-\mathbb{E}^{\pi}\left[\mathbb{P}^{M_t}(B_f)\right]\right|\nonumber\\&=
	\left|\int_S\mathbb{P}^{y}(B_f)(p^{t}(x,dy)-\pi(dy))\right|\nonumber\\
	&\leq d_{TV}(p^{t}(x,dy),\pi(dy)),\end{align} which entails (a).

However, in many situations the processes we deal with  do not meet the properties from (a), (b), (c) and (d). For example, they even do not have to be irreducible   (see Section \ref{s3} for examples of such processes). 
Furthermore, in the discrete-time case,  in \cite{Hernandez-Lerma-Lasserre} it has been shown that if  a Markov process $\textbf{M}$ has a unique invariant probability measure $\pi(dy)$, then either  
\begin{itemize}
	\item [(i)]the relation in \eqref{eq1.4} holds $\pi$-a.e., or
	\item[(ii)]for $\pi$-a.e. $x\in S$, $\pi(dy)$ is singular with respect to $\sum_{t=1}^{\infty}p^{t}(x,dy)$ and $p^{t}(x,dy)$ converges weakly to $\pi(dy)$.
\end{itemize}
   Clearly, in the later case,
$p^{t}(x,dy)$ cannot converge in the total variation metric to $\pi(dy)$. This suggests that the a.s. convergence in (a) is  a too strong property (recall that under the aperiodicity assumption (a) and (d) are equivalent).
Also, the situation in (ii)  suggests that weak convergence might be the key property to  be analysed.
Accordingly, our main aim is to relax the notion of strong ergodicity and, under these new assumptions,   conclude     a relation of the form in \eqref{eq1.2} which holds for all $\mu\in\mathcal{P}(S)$.

In the sequel, assume the following additional structural properties of the state space $(S,\mathcal{S})$.
There exists a metric $d$ on $S$ inducing $\mathcal{S}$ (hence, $\mathcal{S}$ is the Borel $\sigma$-algebra induced by $d$) such that  $(S,d)$ is a Polish space, that is, a complete and separable topological space.  Further, denote by $Lip(S)$ the space of all  $\R$-valued Lipschitz continuous functions on $S$, that is,  functions $f:S\longrightarrow\R$ for which there exists $L_f\geq0$ such that
$$|f(x)-f(y)|\leq L_f\,d(x,y),\qquad x,y\in S.$$ The best
admissible constant
$L_f$
is then denoted by $|f|_{Lip}$.
Also, denote by $C(S)$ the space of all  $\R$-valued continuous functions on $S$ and let $C_b(S):=C(S)\cap B_b(S)$ be the space of all  $\R$-valued continuous and bounded functions on $S$. In this paper we will be concerned with bounded Lipschitz continuous functions only. The reason for that is explained below.
Observe that (i) $f\in Lip(S)\cap B_b(S)$ if, and only if, $f(x)$ is Lipschitz  continuous with respect to $\bar{d}:=d/(1+d)$  (or $1\wedge d$) and (ii) $\bar{d}$ induces the same topology as $d$. Here, $a\wedge b$ denotes the minimum of $a,b\in\R$.   Hence, without loss of generality we may assume that $d$ is bounded, say by $1$, that is, 
$\sup_{x,y\in S}d(x,y)\leq1.$ 
In particular, we have  $Lip(S)\subseteq C_b(S)$.
Now,  recall that the  \emph{Wasserstein metric} of order one (also known as the \emph{Kantorovich-Rubinshtein metric}), denoted by $d_W$, is a metric on  $\mathcal{P}(S)$  defined by $$d_W(\mu(dy),\nu(dy)):=\sup_{f\in Lip(S),\, |f|_{Lip}\leq1}\left|\int_Sf(y)\mu(dy)-\int_Sf(y)\nu(dy)\right|,\qquad \mu,\nu\in\mathcal{P}(S).$$ Observe that, since $d$ is bounded by $1$, $d_W\leq d_{TV}$ (see \cite[Theorem 6.15]{Villani}). In particular, $d_W\leq 1$. Also, if $d$ is the discrete metric on $S$, $d_W= d_{TV}$. Thus, the topology induced by $d_W$ is, in general, finer than the topology induced by $d_{TV}.$ Recall also that $d_W$ is  equivalent to the so-called \emph{modified Kantorovich-Rubinstein metric} and the \emph{L\'evy-Prohorov metric} (see \cite[Theorem 8.10.43]{bogachev}).
A very important property of the Wasserstein metric is that it metrizes the weak convergence of probability measures, which is due to the fact that the underlying metric is bounded. More precisely, a sequence $\{\mu_n(dy)\}_{n\in\N}\subseteq\mathcal{P}(S)$  converges  to $\mu\in\mathcal{P}(S)$ in the Wasserstein topology if, and only if,  $\{\mu_n(dy)\}_{n\in\N}$ converges to  $\mu(dy)$ weakly, that is, $$\lim_{n\to\infty}\int_Sf(y)\mu_n(dy)=\int_Sf(y)\mu(dy),\qquad f\in C_b(S),$$ (see \cite[Corollary 6.13]{Villani}).
For more
on the Wasserstein metric  we refer the readers to
\cite{bogachev} and \cite{Villani}.

We are now in position to state the main results of this paper.
\begin{theorem}\label{tm1.2} Let  $\textbf{M}$ be a Markov process with state space  $(S,\mathcal{S})$. Assume that there is $\pi\in\mathcal{P}(S)$ satisfying
\begin{align}
\label{eq1.6}
\lim_{t\to\infty}\sup_{s\in\mathbb{T}}\int_S d_W(p^{t}(y,dz),\pi(dz))p^{s}(x,dy)=0,\qquad x\in S.\end{align}
Then, for any $p\ge1$, $f\in Lip(S)$ and  $\mu\in\mathcal{P}(S),$ \begin{align}\label{eq1.7}\frac{1}{t}\int_{[0,t)}f(M_s)\tau(ds)\xrightarrow[t\nearrow\infty]{L^p(\Omega,\mathbb{P}^{\mu})} \int_Sf(y)\pi(dy),\end{align} where $\xrightarrow[t\nearrow\infty]{L^p(\Omega,\mathbb{P}^{\mu})}$ denotes the convergence in $L^p(\Omega,\mathbb{P}^{\mu})$.
\end{theorem}
Let us remark that
\begin{itemize}
	\item [(i)] the function $x\longmapsto d_W(p^{t}(x,dy),\pi(dy))$ is $\mathcal{S}/\mathcal{B}(\R)$ measurable for all $t\in\mathbb{T}$ (see the proof of \cite[Lemma 4.5]{Butkovsky} or  \cite[Theorem 4.8]{Hairer-Mattingly-Scheutzow}), hence the relation in \eqref{eq1.6} is well defined.
	\item[(ii)]  the relation in \eqref{eq1.6} is always bounded by $1$.
	\item[(iii)] for any $t\in\mathbb{T}$ and $x\in S$,  $$d_W(p^{t}(x,dy),\pi(dy))\leq\sup_{s\in\mathbb{T}}\int_S d_W(p^{t}(y,dz),\pi(dz))p^{s}(x,dy),$$ thus \eqref{eq1.6} implies \begin{align}\label{eq1.8} \lim_{t\longrightarrow\infty}d_W(p^{t}(x,dy),\pi(dy))=0,\qquad x\in S.\end{align}
	\item[(iv)] \eqref{eq1.6} trivially holds true if  $$\lim_{t\longrightarrow\infty}\sup_{x\in S}d_W(p^{t}(x,dy),\pi(dy))=0$$ (see \cite[Theorems 2.1 and 2.4]{Butkovsky} and \cite[Theorem 5.22]{Chen} for sufficient conditions  ensuring the above relation).
	\item[(v)] we do not require that $\mathbf{M}$ is irreducible  (see Example \ref{e3.1}).
\end{itemize}
As a direct consequence of  \cite[Lemma 3.2 and Proposition 3.12]{Kallenberg} we also conclude the following.
\begin{corollary}\label{c1.3} Let  $\textbf{M}$ be a Markov process with state space  $(S,\mathcal{S})$. Assume that $\textbf{M}$ satisfies the assumptions from Theorem \ref{tm1.2}.
Then,
\begin{itemize}
	\item [(i)]for any  $f\in Lip(S)$ and   $\mu\in\mathcal{P}(S),$ the convergence in \eqref{eq1.7} holds also in probability with respect to $\mathbb{P}^{\mu}(d\omega)$.
	\item[(ii)] for any  $f\in Lip(S)$, $\mu\in\mathcal{P}(S)$ and sequence $\{t_n\}_{n\in\N}\subseteq\mathbb{T}$, $t_n\nearrow\infty$, there is a further subsequence $\{t_{n_k}\}_{k\in\N}\subseteq\{t_{n}\}_{n\in\N}$ (possibly depending on $f(x)$ and $\mu(dy)$) such that  the convergence in \eqref{eq1.7} holds also  $\mathbb{P}^{\mu}$-a.s.
	\item[(iii)] $p^{t}(x,dy)$, $t\in\mathbb{T}$, is  \emph{weak  * mean ergodic}, that is,  $$\frac{1}{t}\int_{[0,t)}\int_{S}p^{t}(x,dy)\mu(dx)$$ converges weakly to $\pi(dy)$ as $t\nearrow\infty$ for every $\mu\in\mathcal{P}(S)$.
\end{itemize}
\end{corollary}

Further, denote by  $C_c(S)$, $C_\infty(S)$ and $C_{b,u}(S)$ the spaces of all $\R$-valued   continuous functions with compact support,  vanishing at infinity and uniformly continuous bounded functions on $S$, respectively.  Then, according to the Stone-Weierstrass theorem (which implies that $Lip(S)\cap C_c(S)$ is dense in $C_c(S)$ with respect to $|\cdot|_\infty$) and \cite[Theorem 1 and Proposition 6]{rumunj} (which state that $Lip(S)$ is dense in $C_{b,u}(S)$ with respect $|\cdot|_\infty$) we have the following.
\begin{corollary}\label{c1.4} Let  $\textbf{M}$ be a Markov process with state space  $(S,\mathcal{S})$. Assume that $\textbf{M}$ satisfies the assumptions from Theorem \ref{tm1.2}.
	Then,
	\begin{itemize}
		\item [(i)]  the relation in \eqref{eq1.7} holds for all $f\in C_c(S)$ and    $\mu\in\mathcal{P}(S).$ In particular, if $S$ is compact, then the relation in \eqref{eq1.7} holds for all $f\in C(S)$ and    $\mu\in\mathcal{P}(S).$
		\item[(ii)] provided $(S,d)$ is also a locally compact space, the relation in \eqref{eq1.7} holds for all $f\in C_\infty(S)$ and    $\mu\in\mathcal{P}(S).$
		\item [(iii)] provided $S$ is also a Hilbert space, the relation in \eqref{eq1.7} holds for all  $f\in C_{b,u}(S)$ and    $\mu\in\mathcal{P}(S).$
	\end{itemize}
\end{corollary}
 Directly from the above corollary we conclude the ergodic $L^p$-version of the Birkhoff ergodic theorem (relation in \eqref{eq1.2}).
 \begin{corollary}\label{c1.7} Let  $\textbf{M}$ be a Markov process with state space  $(S,\mathcal{S})$. In addition to the  assumptions from Theorem \ref{tm1.2}, assume that $\pi(dy)$ is an invariant measure for $\textbf{M}$ and that $(S,d)$ is also a locally compact space. Then,  
 $$\frac{1}{t}\int_{[0,t)}f(M_s)\tau(ds)\xrightarrow[t\nearrow\infty]{L^p(\Omega,\mathbb{P}^{\pi})} \int_Sf(y)\pi(dy)$$ holds for any  $p\ge1$ and $f\in L^p(S,\pi)$.
 \end{corollary}
Observe that the measure $\pi(dy)$ in the previous corollary is a unique invariant measure for $\textbf{M}$. Thus, $\textbf{M}$ is necessarily ergodic.
Further, let us remark here that if $\textbf{M}$ is a strong Feller process, that is, if $x\longmapsto\int_Sf(y)p^{t}(x,dy)$ is a $C_b(S)$ function for every $t\in\mathbb{T}\setminus\{0\}$ and $f\in B_b(S),$ then, under \eqref{eq1.8}, (a) is easily concluded from
the second line in \eqref{eq1.5} and the fact that the Wasserstein metric  metrizes the weak convergence of probability measures. In particular, as we have already commented, (under aperiodicity assumption) this implies   strong ergodicity of $\textbf{M}$. For sufficient conditions which ensure the strong Feller property of Markov processes
see
\cite{Schilling-Wang}. Further, note that, similarly as in the case of the total variation metric, the condition in \eqref{eq1.8}, together with  \cite[Theorem 1.2]{Billingsley},  implies that $\pi(dy)$ is the only measure satisfying \eqref{eq1.8} (and \eqref{eq1.6}). On the other hand, it is not completely clear that \eqref{eq1.8} (or \eqref{eq1.6}) automatically ensures invariance (with respect to $\textbf{M}$) of $\pi(dy)$. However, if, in addition, we assume that   $\textbf{M}$ is a $C_b$-Feller process, that is, if $x\longmapsto\int_Sf(y)p^{t}(x,dy)$ is a $C_b(S)$ function for every $t\in\mathbb{T}$ and $f\in C_b(S),$ then \eqref{eq1.8} implies  invariance of $\pi(dy)$. Indeed, for any $t\in\mathbb{T}$, $x\in S$ and $f\in C_b(S),$ we have that
\begin{align*}\int_{S}f(y)\pi(dy)&=\lim_{s\to\infty}\int_Sf(y)p^{s+t}(x,dy)\\&=\lim_{s\to\infty}\int_{S}\int_{S}f(y)p^{t}(z,dy)p^{s}(x,dz)\\&=\int_Sf(y)\int_Sp^{t}(z,dy)\pi(dz).\end{align*} The assertion now follows from  \cite[Theorem 1.2]{Billingsley}. For conditions ensuring that a
Markov process is a
$C_b$-Feller process  see \cite{Schilling}.
Another condition ensuring that  \eqref{eq1.8} (or \eqref{eq1.6}) implies  invariance of $\pi(dy)$ is contractivity of $p^{t}(x,dy)$ with respect to $d_W$, that is, \begin{align} \label{eq1.9}
d_W(p^{t}(x_1,dy),p^{t}(x_2,dy))\leq d(x_1,x_2),\qquad t\in\mathbb{T},\ x_1,x_2\in S.
\end{align}  To see this, first, according to  the proof of \cite[Lemma 4.5]{Butkovsky}, the above relation yields that for any $t\in\mathbb{T}$ and $\mu, \nu\in\mathcal{P}(S)$, we have that \begin{align} \label{eq1.10}d_W\left(\int_Sp^{t}(x,dy)\mu(dx),\int_Sp^{t}(x,dy)\nu(dx)\right)\leq d_W(\mu(dy),\nu(dy)).\end{align}
In particular, by employing \eqref{eq1.8}, we have that $$\lim_{s\to\infty}d_W\left(\int_Sp^{t}(x,dy)\pi(dx),p^{s+t}(x,dy)\right)\leq \lim_{s\to\infty}d_W(\pi(dy),p^{s}(x,dy))=0,\qquad t\in\mathbb{T},\ x\in S,$$ which yields,
$$\int_Sf(y)\int_Sp^{t}(x,dy)\pi(dx)=\lim_{s\to\infty}\int_Sf(y)p^{s+t}(x,dy)=\int_{S}f(y)\pi(dy),\qquad t\in\mathbb{T},\ x\in S.$$ Recall that  every Markov semigroup is  contractive with respect to $d_{TV}$ (in the sense \eqref{eq1.10}). Finally, if $\pi(dy)$ is an invariant probability measure of $\textbf{M}$ satisfying \eqref{eq1.8}, then,   just by applying \cite[Theorem 1.2]{Billingsley}, we easily see that
$\pi(dy)$ is actually  a unique invariant probability measure for $\textbf{M}$. Consequently, $\textbf{M}$ is ergodic.

Based on the previous discussions, it is tempting to conclude that  \eqref{eq1.6} in Theorem \ref{tm1.2} might be replaced by \eqref{eq1.8} (as in the strong ergodicity situation).   However, it is not completely clear that \eqref{eq1.8} alone is sufficient to conclude the assertion of Theorem \ref{tm1.2}.
In the following theorem we prove that, in addition to \eqref{eq1.8},
 if the transition function $p^{t}(x,dy)$, $t\in\mathbb{T}$, is Lipschitz, that is, $x\longmapsto \int_{S}f(y)p^{t}(x,dy)$ is Lipschitz for every $f\in Lip(S)$ and $t\in\mathbb{T}$, then the assertion of Theorem \ref{tm1.2} holds true.
For more on Markov processes with Lipschitz transition function  we refer the readers to \cite{bass1}.
\begin{theorem} \label{tm1.5} Let  $\textbf{M}$ be a Markov process with state space  $(S,\mathcal{S})$. Assume that there is $\pi\in\mathcal{P}(S)$ satisfying \eqref{eq1.8}, and for every $f\in Lip(S)$  and $t\in\mathbb{T}$, the function $$F_{f,t}(x):=\int_Sf(y)p^{t}(x,dy),\qquad x\in S,$$ is also in $Lip(S)$ with $|F_{f,t}|_{Lip}\leq C_f,$ where the constant $C_f$ depends only on $f(x)$.
	Then, for any $f\in Lip(S)$ and  $\mu\in\mathcal{P}(S),$ $\textbf{M}$ satisfies the relation in \eqref{eq1.7}.
\end{theorem}
As we have already commented, one feature that distinguishes the total variation metric  among other metrics (on $\mathcal{P}(S)$) is that, for any Markov transition function, one always has the contraction property in \eqref{eq1.10}. On the other hand, the Wasserstein metric  may not be contracting in general. It is therefore natural to  focus only on Wasserstein metrics that are contracting for $\textbf{M}$ (in the sense of relation \eqref{eq1.9}).
Also, let us remark that it has been observed in \cite{Hairer-Mattingly-Scheutzow} that \eqref{eq1.9} is  essential (but not sufficient) in obtaining the convergence to an unique invariant probability measure (see \cite{Hairer-Mattingly-Scheutzow} for detailed discussion on this property). As a direct consequence of the contraction property in  \eqref{eq1.9} we get that
\begin{itemize}
	\item [(i)] the function $t\longmapsto d_W(p^{t}(x,dy),\pi(dy))$ is non-increasing for all $x\in S.$
	\item[(ii)] for every $f\in Lip(S)$  and $t\in\mathbb{T}$,  $F_{f,t}\in Lip(S)$ and $|F_{f,t}|_{Lip}\leq|f|_{Lip}.$
\end{itemize}
Thus, as a direct consequence of Theorem \ref{tm1.5} we get the following.
\begin{theorem}\label{tm1.6} Let  $\textbf{M}$ be a Markov process with state space  $(S,\mathcal{S})$. Assume that there is $\pi\in\mathcal{P}(S)$ satisfying \eqref{eq1.8} and
	\eqref{eq1.9} (hence, as we have already commented, $\pi(dy)$ is necessarily a unique invariant measure  for $\textbf{M}$).
	Then, for any $f\in Lip(S)$ and  $\mu\in\mathcal{P}(S),$ $\textbf{M}$ satisfies the relation in \eqref{eq1.7}. 
\end{theorem}

\section{Proof of Theorems \ref{tm1.2} and \ref{tm1.5} and Corollary  \ref{c1.7}}\label{s2}
In this section, we prove Theorems \ref{tm1.2} and \ref{tm1.5} and   Corollary  \ref{c1.7}. Before the proofs, we introduce some notation that will be used in the sequel. For $\mu\in\mathcal{P}(S)$ and  $f\in B_b(S)$, we  write $\mu(f)$ for $\int_Sf(y)\mu(dy).$ Also, with $\processt{P}$ is denoted the semigroup of $\textbf{M}$ on $B_b(S)$, that is, $P_tf(x):=\int_Sf(y)p^{t}(x,dy),$ $t\in\mathbb{T}$, $x\in S$, $f\in B_b(S).$

\begin{proof}[Proof of Theorem \ref{tm1.2}]
First, observe that it suffices  to prove the assertion only for the Dirac measures as the initial distributions of $\textbf{M}$. Also, since for each $p\ge1$ and $f\in Lip(S)$ (recall that $Lip(S)\subseteq C_b(S)$) the family $$\left\{\left|\frac{1}{t}\int_{[0,t)}f(M_s)\tau(ds)\right|^p:t\in\mathbb{T}\right\}$$ is uniformly integrable,  it suffices to consider the case when $p=2$ only (see \cite[Proposition 3.12]{Kallenberg}). 
For arbitrary $x\in S$ and $f\in Lip(S)$   we have that
\begin{align*}&
\mathbb{E}^{x}\left[\left(\frac{1}{t}\int_{[0,t)}f(M_s)\tau(ds)-\pi(f)\right)^{2}\right]\nonumber\\
&=\mathbb{E}^{x}\left[\left(\frac{1}{t}\int_{[0,t)}f(M_s)\tau(ds)\right)^{2}\right]-\frac{2\pi(f)}{t}\int_{[0,t)}P_sf(x)\tau(ds)+\pi(f)^{2}.\end{align*}
Assume first that $\mathbb{T}=[0,\infty).$ Thus,  
$$\mathbb{E}^{x}\left[\left(\frac{1}{t}\int_{[0,t)}f(M_s)\tau(ds)-\pi(f)\right)^{2}\right]=\mathbb{E}^{x}\left[\left(\int_{0}^{1}f(M_{st})ds\right)^{2}\right]-2\pi(f)\int_{0}^{1}P_{st}f(x)ds+\pi(f)^{2}.$$
In the sequel, without loss of generality, assume that $|f|_{Lip}\leq1$. Otherwise, divide $f(x)$ by $|f|_{Lip}$.
Next, note that
$$\lim_{t\to\infty}\int_{0}^{1}P_{st}f(x)ds=\pi(f).$$
Indeed, by assumption, $$\lim_{t\to\infty}d_W(p^{t}(x,dy),\pi(dy))\leq \lim_{t\to\infty}\sup_{s\in\mathbb{T}}P_s\, d_W(p^{t}(\cdot,dz),\pi(dz))(x)=0.$$
  In particular, $$\lim_{t\to\infty}|P_tf(x)-\pi(f)|=0.$$ Thus, the claim is a direct consequence of the dominated convergence theorem.
 Note  that  the above relation also
 holds  for all $x\in S$ and $f\in C_b(S)$
 (the Wasserstein metric metrizes the weak convergence of
probability measures).
Consequently, in order to prove the assertion,  it suffices to prove that
$$\lim_{t\to\infty}\mathbb{E}^{x}\left[\left(\int_{0}^{1}f(M_{st})ds\right)^{2}\right]=\pi(f)^{2}.$$
By Fubini's theorem and the Markov property
 we have that
\begin{align*}&\left|\mathbb{E}^{x}\left[\left(\int_{0}^{1}f(M_{st})ds\right)^{2}\right]-\pi(f)^{2}\right|\\&=\left|2\mathbb{E}^{x}\left[\int_{0}^{1}\int_{0}^{s}f(M_{st})f(M_{ut})duds\right]-\pi(f)^{2}\right|\\
&=\left|2\int_{0}^{1}\int_{0}^{s}P_{ut}(f\,\mathbb{E}^{\cdot}[f(M_{(s-u)t})])(x)duds-\pi(f)^{2}\right|\\
&\leq2\left|\int_{0}^{1}\int_{0}^{s}\left(P_{ut}(fP_{(s-u)t}f)(x)-P_{ut}(f\pi(f))(x)\right)duds\right|+ 2\left|\int_{0}^{1}\int_{0}^{s}\left(P_{ut}(f\pi(f))(x)-\pi(f)^{2}\right)duds\right|\\
&=2\left|\int_{0}^{1}\int_{0}^{s}P_{ut}\left(P_{(s-u)t}f-\pi(f)\right)f(x)duds\right|+ 2\left|\pi(f)\int_{0}^{1}\int_{0}^{s}(P_{ut}f(x)-\pi(f))duds\right|
\\
&\leq 2|f|_{\infty}\int_{0}^{1}\int_{0}^{s}P_{ut}|P_{(s-u)t}f-\pi(f)|(x)duds +2|f|_{\infty}\int_{0}^{1}\int_{0}^{s}|P_{ut}f(x)-\pi(f)|duds\\
&\leq 2|f|_{\infty}\int_{0}^{1}\int_{0}^{s}\sup_{v\in[0,\infty)}P_{v}d_W(p^{(s-u)t}(\cdot,dz),\pi(dz))(x)duds+2|f|_{\infty}\int_{0}^{1}\int_{0}^{s}|P_{ut}f(x)-\pi(f)|duds.
\end{align*}  Now, by letting $t\nearrow\infty$ the assertion follows. 

In the discrete-time case we have
$$\frac{1}{t}\sum_{n=0}^{t-1}f(M_n)=\frac{1}{t}\int_0^{t}f(M_{[s]})ds=\int_0^{1}f(M_{[st]})ds, \qquad t\in\{1,2,3,\ldots\},$$
where $[a]$ denotes the integer part of $a\in\R$.
Hence,  $$\mathbb{E}^{x}\left[\left(\frac{1}{t}\int_{[0,t)}f(M_s)\tau(ds)-\pi(f)\right)^{2}\right]=\mathbb{E}^{x}\left[\left(\int_{0}^{1}f(M_{[st]})ds\right)^{2}\right]-2\pi(f)\int_{0}^{1}P_{[st]}f(x)ds+\pi(f)^{2}.$$
 Now, the proof of the assertion proceeds analogously as in the continuous-time case.
\end{proof}

\begin{proof}[Proof of Corollary \ref{c1.7}] Let $f\in L^p(S,\pi)$, $p\ge1$. Then, by \cite[Proposition 7.9]{Folland}, for any $\varepsilon>0$ there is $f_\varepsilon\in C_c(S)$ such that $$\left(\int_{S}|f(x)-f_\varepsilon(x)|^{p}\pi(dx)\right)^{1/p}<\frac{\varepsilon}{2}.$$ Further,   by Fubini's theorem and Jensen's inequality, we have that
	\begin{align*}&\left(\mathbb{E}^\pi\left[\left|\frac{1}{t}\int_{[0,t)}f(M_s)\tau(ds)-\pi(f)\right|^{p}\right]\right)^{1/p}\\
	&\le\left(\mathbb{E}^\pi\left[\left|\frac{1}{t}\int_{[0,t)}f(M_s)\tau(ds)-\frac{1}{t}\int_{[0,t)}f_\varepsilon(M_s)\tau(ds)\right|^{p}\right]\right)^{1/p}\\ &\ \ \  +\left(\mathbb{E}^\pi\left[\left|\frac{1}{t}\int_{[0,t)}f_\varepsilon(M_s)\tau(ds)-\pi(f_\varepsilon)\right|^{p}\right]\right)^{1/p}+\left(\mathbb{E}^\pi\left[\left|\pi(f_\varepsilon)-\pi(f)\right|^{p}\right]\right)^{1/p}\end{align*}
	\begin{align*}
	&\le \left(\frac{1}{t}\int_{[0,t)}\mathbb{E}^\pi\left[|f(M_s)-f_\varepsilon(M_s)|^p\right]\tau(ds)\right)^{1/p} +\left(\mathbb{E}^\pi\left[\left|\frac{1}{t}\int_{[0,t)}f_\varepsilon(M_s)\tau(ds)-\pi(f_\varepsilon)\right|^{p}\right]\right)^{1/p}\\ &\ \ \ +\left(\int_S|f_\varepsilon(x)-f(x)|^p\pi(dx)\right)^{1/p}\\
	&<\left(\int_S|f(x)-f_\varepsilon(x)|^p\pi(dx)\right)^{1/p} +\left(\mathbb{E}^\mu\left[\left|\frac{1}{t}\int_{[0,t)}f_\varepsilon(M_s)\tau(ds)-\pi(f_\varepsilon)\right|^{p}\right]\right)^{1/p}+\frac{\varepsilon}{2}\\
	&<\varepsilon+\left(\mathbb{E}^\mu\left[\left|\frac{1}{t}\int_{[0,t)}f_\varepsilon(M_s)\tau(ds)-\pi(f_\varepsilon)\right|^{p}\right]\right)^{1/p}.
	\end{align*} Finally, by employing Corollary \ref{c1.4} (i) we conclude 
	 $$\lim_{t\to\infty}\left(\mathbb{E}^\mu\left[\left|\frac{1}{t}\int_{[0,t)}f(M_s)\tau(ds)-\pi(f)\right|^{p}\right]\right)^{1/p}<\varepsilon,$$ which proofs the desired result.
	\end{proof}

\begin{proof}[Proof of Theorem \ref{tm1.5}]
	We proceed similarly as in Theorem \ref{tm1.2}. We discuss only the continuous-time case. Again, the assertion will follow if we prove
 that
	$$\lim_{t\to\infty}\mathbb{E}^{x}\left[\left(\int_{0}^{1}f(M_{st})ds\right)^{2}\right]=\pi(f)^{2},\qquad x\in S.$$
	By assumption, $$\lim_{t\to\infty}d_W(p^{t}(x,dy),\pi(dy))=0\qquad\textrm{and}\qquad \lim_{t\to\infty}\int_S d_W(p^{t}(y,dz),\pi(dz))\pi(dy)=0.$$    Recall that $\sup_{t\in\mathbb{T},\,x\in S}d_W(p^{t}(x,dy),\pi(dy))\leq1$,  hence the later condition is a consequence of the dominated convergence theorem.
	In particular,   $$\lim_{t\to\infty}|P_tf(x)-\pi(f)|=0.$$   Again, similarly as before, by Fubini's theorem and the Markov property,
	we have \begin{align*}&\left|\mathbb{E}^{x}\left[\left(\int_{0}^{1}f(M_{st})ds\right)^{2}\right]-\pi(f)^{2}\right|\\&\leq2\left|\int_{0}^{1}\int_{0}^{s}P_{ut}\left(P_{(s-u)t}f-\pi(f)\right)f(x)duds\right| + 2\left|\pi(f)\int_{0}^{1}\int_{0}^{s}(P_{ut}f(x)-\pi(f))duds\right|\\
&\le2\left|\int_{0}^{1}\int_{0}^{s}P_{ut}\left(P_{(s-u)t}f-\pi(f)\right)f(x)duds\right|+2|f|_\infty\int_{0}^{1}\int_{0}^{s}|P_{ut}f(x)-\pi(f)|duds\\ &\leq2|f|_\infty\int_{0}^{1}\int_{0}^{s}|P_{ut}f(x)-\pi(f)|duds+2\left|\int_{0}^{1}\int_{0}^{s}\left(P_{ut}\left(fP_{(s-u)t}f\right)(x)-\pi(fP_{(s-u)t}f)\right)duds\right|\\&\ \ \ 
 +2\left|\int_{0}^{1}\int_{0}^{s}\left(\pi\left(fP_{(s-u)t}f\right)-\pi(f)P_{ut}f(x)\right)duds\right|.\end{align*}
 Now, since
 $Lip(S)\subseteq C_b(S)$, $P_tf\in Lip(S)$ (with $|P_tf|_{Lip}\leq C_f$) for every $f\in Lip(S)$ and $t\in\mathbb{T}$, and for $f,g\in Lip(S)$, $fg\in Lip(s)$ (with $|fg|_{Lip}\leq|f|_\infty|f|_{Lip}+|g|_\infty|g|_{Lip}$), we conclude that
  $$\lim_{t\to\infty}\left|\mathbb{E}^{x}\left[\left(\int_{0}^{1}f(M_{st})ds\right)^{2}\right]-\pi(f)^{2}\right|\le2\lim_{t\to\infty}\left|\int_{0}^{1}\int_{0}^{s}\left(\pi\left(fP_{(s-u)t}f\right)-\pi(f)P_{ut}f(x)\right)duds\right|.$$
 Finally, we have
	\begin{align*}
	&2\left|\int_{0}^{1}\int_{0}^{s}\left(\pi\left(fP_{(s-u)t}f\right)-\pi(f)P_{ut}f(x)\right)duds\right|\\& \leq2\left|\int_{0}^{1}\int_{0}^{s}\left(\pi\left(fP_{(s-u)t}f\right)-\pi(f)^{2}\right)duds\right| +2\left|\int_{0}^{1}\int_{0}^{s}\left(\pi(f)^{2}-\pi(f)P_{ut}f(x)\right)duds\right|\\ &\leq2\left|\int_{0}^{1}\int_{0}^{s}\int_S\left(P_{(s-u)t}f(y)-\pi(f)\right)f(y)\pi(dy)duds\right| +
2|f|_{\infty}\int_{0}^{1}\int_{0}^{s}\left|\pi(f)-P_{ut}f(x)\right|duds\\
&=2|f|_\infty\int_{0}^{1}\int_{0}^{s}\int_S\left|P_{(s-u)t}f(y)-\pi(f)\right|\pi(dy)duds+2|f|_{\infty}\int_{0}^{1}\int_{0}^{s}\left|\pi(f)-P_{ut}f(x)\right|duds,
\end{align*} which concludes the proof.
\end{proof}

\section{Examples}\label{s3}
In this section, we give some applications of Theorems \ref{tm1.2},  \ref{tm1.5} and  \ref{tm1.6}.

\begin{example}\label{e3.1}{\rm Let $\{X_n\}_{n\geq1}$ be a sequence of i.i.d. $\R$-valued random variables on a probability space $(\Omega,\mathcal{F},\mathbb{P})$, satisfying $\mathbb{P}(X_n=0)=\mathbb{P}(X_n=1/2)=1/2$. Define $$M_{n+1}:=\frac{1}{2}M_n+X_{n+1},\qquad n\geq0,\ M_0\in[0,1].$$
Clearly, $\mathbf{M}:=\chain{M}$ is a Markov process with state space $([0,1],\mathcal{B}([0,1]))$ and transition function $p(x,dy):=\mathbb{P}(X_1+x/2\in dy),$ $x\in[0,1]$. Also, it is easy to see that the Lebesgue measure (on $\mathcal{B}([0,1])$) is invariant for $\mathbf{M}$ and $\mathbf{M}$ is ergodic with respect to ${\rm Leb}(dy)$.
In particular, the relation in \eqref{eq1.2} holds true. However, observe that, since ${\rm Leb}(dy)$ is singular with respect to $p(x,dy)$,  $\mathbf{M}$ is not strongly ergodic.
Let $d(x,y):=|x-y|,$ $x,y\in[0,1].$ Now, by a straightforward computation, we get $$d_W(p(x_1,dy),p(x_2,dy))\leq\frac{d(x_1,x_2)}{2},\qquad x_1,x_2\in[0,1],$$ which, together with the proof of \cite[Lemma 4.5]{Butkovsky}, yields that for all $n\geq1$ and  $x,x_1,x_2\in[0,1]$, $$d_W(p^{n}(x_1,dy),p^{n}(x_2,dy))\leq\frac{d(x_1,x_2)}{2^{n}}\qquad\textrm{and}\qquad d_W(p^{n}(x,dy),{\rm Leb}(dy))\leq\frac{1}{2^{n}}.$$
Thus, Theorem \ref{tm1.2} (and  Theorems   \ref{tm1.5} and \ref{tm1.6}) applies. Also, note that $\mathbf{M}$ is not irreducible.}
\end{example}

\begin{example}\label{e3.2}{\rm (\cite[Section 3]{Hairer-Mattingly-Scheutzow} and \cite[Example 3.4]{Butkovsky}).
 Let $\mathcal{C}:=C([-1,0],\R)$ be the space of all continuous functions from $[-1,0]$ to $\R$ endowed with the supremum norm. For $t\geq0$ and a function $f(s)$ defined on $[t-1,t]$, we write $f^{t}(s):=f(s+t)$, $s\in[-1,0]$.
 Consider the following stochastic functional
differential equation (the so-called \emph{stochastic delay equation})
\begin{align}\label{eq3.1}dM_t=-M_tdt+G(M_{t-1})dB_t,\qquad  t\geq0,\ M^{0}\in \mathcal{C},\end{align} where $G:\R\longrightarrow\R$ and $\process{B}$ is a standard one-dimensional Brownian motion. Now, by assuming that
 $G(u)$ is bounded, Lipschitz continuous, strictly positive and strictly increasing, in \cite[Theorems 3.1 and 4.8]{Hairer-Mattingly-Scheutzow}  it has been proven that the equation in \eqref{eq3.1} admits a unique strong solution $\mathbf{M}:=\{M^{t}\}_{t\geq0}$ which is a strong Markov and $C_b$-Feller  process with state space $(\mathcal{C},\mathcal{B}(\mathcal{C}))$.
Furthermore, due to \cite[Corollary 4.11]{Hairer-Mattingly-Scheutzow}, $\mathbf{M}$
    possesses a unique invariant probability measure $\pi(dy)$ which is singular with respect to the transition function $p^{t}(x,dy)$ of $\mathbf{M}$. Hence, $\mathbf{M}$ is not strongly ergodic.  Next, fix $\delta>0$ and define $d(x,y):=1\wedge |x-y|_\infty/\delta$, $x,y\in\mathcal{C}$.
    Then, according to \cite[Theorem 4.8]{Hairer-Mattingly-Scheutzow} (see also \cite[Example 3.2]{Butkovsky}), for suitable (small enough) constant $\delta>0$ there exist a constant $c>0$ and Borel function $C:\mathcal{C}\longrightarrow[0,\infty)$,  such that \eqref{eq1.9} holds, $$d_W(p^{t}(x,dy),\pi(dy))\leq C(x)e^{-ct}\qquad\textrm{and}\qquad \sup_{s\geq 0}P_sC(x)<\infty,\quad t\geq0,\ x\in\mathcal{C},$$ hence we are in the situation of Theorem \ref{tm1.2} (and Theorems    \ref{tm1.5} and \ref{tm1.6}).
}\end{example}

 \section*{Acknowledgement} This work has been supported in part by the Croatian Science Foundation under Project 3526 and  NEWFELPRO Programme  under Project 31. The author  thanks the anonymous reviewer
for careful reading of the paper and for helpful comments that led to improvement in the
presentation.

\bibliographystyle{plain}
\bibliography{References}

\end{document}